\documentclass{amsart}
\usepackage{cases}
\usepackage{latexsym}
\usepackage[arrow,matrix]{xy}
\usepackage{stmaryrd}
\usepackage{amsfonts}
\usepackage{amsmath,amssymb,amscd,bbm,amsthm,mathrsfs,dsfont}
\usepackage{fancyhdr}
\usepackage{amsxtra,ifthen}
\usepackage{verbatim}
\usepackage{color}

\numberwithin{equation}{section}

\theoremstyle{plain}
\newtheorem{prop}{Proposition}[section]
\newtheorem{thm}[prop]{Theorem}
\newtheorem{cor}[prop]{Corollary}
\newtheorem{lem}[prop]{Lemma}
\theoremstyle{definition}
\newtheorem{dfn}[prop]{Definition}

\newtheorem{example}[prop]{Example}
\newtheorem{remark}[prop]{Remark}

\newcommand{\lam}{\lambda}

\DeclareMathOperator{\rad}{rad}
\DeclareMathOperator{\rank}{rank}
\DeclareMathOperator{\End}{End}

\begin{document}
%\vspace*{-20mm}
\title[Projective cell modules of Frobenius cellular algebras ]{Projective cell modules of Frobenius cellular algebras}

\author{Yanbo Li and Deke Zhao$^\ast$}

\address{Li: School of Mathematics and Statistics, Northeastern
University at Qinhuangdao, Qinhuangdao, 066004, P.R. China}

\email{liyanbo707@163.com}

\address{Zhao: School of Applied Mathematics, Beijing Normal University at Zhuhai, Zhuhai, 519085, P. R. China}

\email{deke@amss.ac.cn}

\begin{abstract}
For a finite dimensional Frobenius cellular algebra, a sufficient
and necessary condition for a simple cell module to be projective is
given. A special case that dual bases of the cellular basis
satisfying a certain condition is also considered. The result is
similar to that in symmetric case.
\end{abstract}



\thanks{$^\ast$ Corresponding author.}

\thanks{The work of the first author is supported by Fundamental Research
Funds for the Central Universities (N110423007). The work of the
second author is supported by the National Natural Science
Foundation of China (No.11101037).}

\maketitle

\section{Introduction}\label{xxsec1}

Cellular algebras were introduced by Graham and Lehrer \cite{GL} in
1996, motivated by previous work of Kazhdan and Lusztig \cite{KL}.
They were defined by a so-called cellular basis satisfying certain
axioms. The theory of cellular algebras provides a systematic
framework for studying the representation theory of many interesting
algebras from mathematics and physics. The classical examples of
cellular algebras include Hecke algebras of finite type, Ariki-Koike
algebras, Brauer algebras, Birman-Wenzl algebras and so on. We refer
the reader to \cite{G, GL, Xi2} for details. Recently, Koenig and Xi
\cite{KX8} introduced affine cellular algebras which contain
cellular algebras as special cases. They proved affine Hecke
algebras of type A are affine cellular.

Several important classes of cellular algebras are symmetric, such
as Hecke algebras of finite types, Ariki-Koike algebras and so on.
In \cite{L, L2}, Li studied the general theory of symmetric cellular
algebras, such as dual bases, centres and radicals. Moreover, Li and
Xiao considered the classification of the projective cell modules of
symmetric cellular agebras in \cite{LX}. Frobenius algebras are
natural generalizations of symmetric algebras. It is well known that
a Frobenius algebra is symmetric if and only if there exists a
Nakayama automorphism being the identical mapping. Examples of
non-symmetric Frobenius cellular algebras could be found in
\cite{KX2}. In \cite{L3}, Li investigated Nakayama automorphisms of
Frobenius cellular algebras. It is proved that the matrix associated
with a Nakayama automorphism with respect to a cellular basis is
uni-triangular under a certain condition. A natural question is
which properties of symmetric cellular algebras can be generalized
to Frobenius cases. In this paper, for a Frobenius cellular algebra,
we will give a sufficient and necessary condition for a simple cell
module to be projective. A special case that dual bases of the
cellular basis satisfying a certain condition is also considered.
The result is similar to that in symmetric case.

The paper is organized as follows. We begin with a quick review on
the theory of cellular algebras and Frobenius algebras. In
particular, we give in this section a corollary of
Gasch\"{u}tz-Ikeda's Theorem, which is more useful to study cell
modules of Frobenius cellular algebras. Then in Section 3, after
giving some properties of Frobenius cellular algebras, we give a
sufficient and necessary condition for a simple cell module to be
projective. A special case will be considered in Section 4. The
result is similar to that in symmetric case.

\section{Preliminaries}\label{xxsec2}

In this section, we will establish the basic notations and some
known results which are needed in the following sections. The main
references for this section are \cite{GP} and \cite{GL}.

\subsection{Frobenius algebras}
Let $R$ be a commutative ring with identity and $A$  a finite dimensional  associative
$R$-algebra. Suppose that there exists an $R$-bilinear form
$f:A\times A\rightarrow R.$ We say $f$ is
{\it non-degenerate} if the determinant of the matrix
$(f(a_{i},a_{j}))_{a_{i},a_{j}\in B}$ is not zero for some basis $B$
of $A$. We call $f$ {\it associative} if $f(ab,c)=f(a,bc)$ for all
$a,b,c\in A$.
\begin{dfn}\label{2.1}
An $R$-algebra $A$ is called {\it Frobenius} if there is a non-degenerate
associative bilinear form $f$ on $A$.
\end{dfn}

Let $A$ be a Frobenius algebra with a basis $B=\{a_{i}\mid
i=1,\ldots,n\}$. Let us take a non-degenerate associative bilinear
form $f$. Define an $R$-linear map $\tau: A\rightarrow R$ by
$$\tau(a)=f(a,1).$$
We call $d=\{d_{i}\mid i=1,\ldots,n\}$ the {\it right dual basis} of $B$
which is uniquely determined by the requirement that
$\tau(a_{i}d_{j})=\delta_{ij}$ for all $i, j=1,\ldots,n$. Similarly,
the {\it left dual basis} $D=\{D_{i}\mid i=i,\ldots,n\}$ is determined by
the requirement that $\tau(D_{j}a_{i})=\delta_{ij}$. Define an
$R$-linear map $\alpha:A\rightarrow A$ by $$\alpha(d_{i})=D_{i}.$$
Then $\alpha$ is a Nakayama automorphism of $A$.

\smallskip

From now on, all of modules considered in this paper will always be
left modules. Let $M$ be an $A$-module and $\theta\in {\rm
End}_R(M)$. Then we define the {\it averaging operator} $I(\theta)\in {\rm
End}_R(M)$ by
$$I(\theta)(m):=\sum_ia_i\theta(D_im) \quad\quad \forall\, m\in M.$$ It
is an $A$-module endomorphism. Furthermore, $I(\theta)$ is
independence of the choice of basis. This implies that $I(\theta)$
also can be defined as follows.
$$I(\theta)(m)=\sum_id_i\theta(a_im) \quad\quad \forall\, m\in M.$$
Let $\theta, \pi, \varphi\in \End_A(M)$. Then the definition implies
that $$I(\varphi\circ \theta)=I(\varphi)\circ \theta \quad {\rm
and}\quad I(\theta\circ\pi)=\theta\circ I(\pi).$$ Moreover, it is
helpful to point out that $I(\theta)\in \End_A(M)$.

One of the importance of the averaging operator is that it provides
a criterion for an $A$-module being projective.

\begin{lem}\rm (Gasch\"{u}tz-Ikeda)\label{2.2}
Let $A$ be a Frobenius algebra and $M$ an $A$-module. Then $M$ is
projective if and only if there exists some $\varphi\in {\rm
End}_R(M)$ such that $I(\varphi)={\rm id}_M$.
\end{lem}

Let $R$ be a field. Suppose that $\dim M=m$ and take a basis $\{v_1,
\cdots, v_m\}$. For $1\leq i, j\leq m$, define $\varphi_{ij}\in
\End_R(M)$ by $\varphi_{ij}(v_s)=\delta_{is}v_j$. Clearly,
$\varphi_{ij}$ form a basis of $\End_R(M)$. The following result is
a simple corollary of Lemma \ref{2.2}. However, it is important in
this paper. Hence we will give a complete proof of it here.

\begin{cor}\label{2.3}
Let $A$ be a finite dimensional Frobenius algebra and $M$ a simple
$A$-module. Then $M$ is projective if and only if there exists some
$\varphi_{ij}$ such that $I(\varphi_{ij})\neq 0$.
\end{cor}

\begin{proof}
Since $\varphi_{ij}$ form a basis of $\End_R(M)$, the necessity is a
direct corollary of Lemma \ref{2.2}. Conversely, assume that there
exists some $\varphi_{ij}$ such that $I(\varphi_{ij})\neq 0$. Note
that $M$ is simple. Then Schur's lemma implies that
$I(\varphi_{ij})=r_{ij}\,{\rm id_M}$, where $r_{ij}\in R-\{0\}$. Let $\pi:
N\rightarrow M$ be an epimorphism of $A$-modules. Then there exists
an $R$-linear map $\mu: M\rightarrow N$ such that $\pi\circ\mu={\rm
id_M}$. This gives that $\pi\circ\mu\circ\varphi_{ij}=\varphi_{ij}$
and thus $I(\pi\circ\mu\circ\varphi_{ij})=r_{ij}\,{\rm id_M}$. This
implies that $r_{ij}^{-1}I(\mu\circ\varphi_{ij})$ is the desired
map.
\end{proof}

\subsection{Cellular algebras}

Let us recall the definition of a cellular algebra given by Graham
and Lehrer in \cite{GL}.

\begin{dfn} {\rm \cite{GL}}\label{2.4}
Let $R$ be a commutative ring with identity. An
associative unital $R$-algebra is called a {\it cellular algebra} with
cell datum $(\Lambda, M, C, i)$ if the following conditions are
satisfied:

{\rm(C1)} The finite set $\Lambda$ is a poset. Associated with each
$\lam\in\Lambda$, there is a finite set $M(\lam)$. The algebra $A$
has an $R$-basis $\{C_{S,T}^\lam \mid \lam\in\Lambda,\,\, S,T\in
M(\lam)\}$.

\smallskip

{\rm(C2)} The map $i$ is an $R$-linear anti-automorphism of $A$ with
$i^{2}={\rm id}$ which sends $C_{S,T}^\lam$ to $C_{T,S}^\lam$ for
all $\lam\in\Lambda$ and $S,T\in M(\lam)$.

\smallskip

{\rm(C3)} If $\lam\in\Lambda$ and $S,T\in M(\lam)$, then for any
element
$a\in A$, we have\\
$$aC_{S,T}^\lam\equiv\sum_{S^{'}\in
M(\lam)}r_{a}(S',S)C_{S^{'},T}^{\lam} \,\,\,\,(\rm {mod}\,\,\,
 A(<\lam)),$$ where $r_{a}(S^{'},S)\in R$ is independent of $T$ and where $A(<\lam)$ is the
$R$-submodule of $A$ generated by $\{C_{U,V}^\mu \mid U, V\in
M(\mu),\,\,\mu<\lam\}$.

\smallskip

If we apply $i$ to the equation in {\rm(C3)}, we obtain

{$\rm(C3')$} $C_{T,S}^\lam i(a)\equiv\sum\limits_{S^{'}\in
M(\lam)}r_{a}(S^{'},S)C_{T,S^{'}}^{\lam} \,\,\,\,(\mod A(<\lam)).$
\end{dfn}

As a natural consequence of the axioms, the {\em cell modules} are
defined as follows.

\begin{dfn} \label{2.5}
Let $A$ be a cellular algebra with cell datum $(\Lambda, M, C, i)$.
For each $\lam\in\Lambda$, the {\it cell modules} $W_C(\lam)$ is left $A$-module defined as
follows: $W_C(\lam)$ is a free $R$-module with basis $\{C_{S}\mid
S\in M(\lam)\}$ and $A$-action
defined by\\
$$aC_{S}=\sum_{S^{'}\in M(\lam)}r_{a}(S^{'},S)C_{S^{'}}
\,\,\,\,(a\in A,S\in M(\lam)),$$ where $r_{a}(S^{'},S)$ is the
element of $R$ defined in Definition {\rm\ref{2.4}}{\rm(C3)}.
\end{dfn}

Let $A$ be a cellular algebra with cell datum $(\Lambda, M, C, i)$.
Let $\lam\in\Lambda$ and $S,T,U,V\in M(\lam)$. Then we have from
Definition \ref{2.4} that
$$C_{S,T}^\lam C_{U,V}^\lam \equiv\Phi(T,U)C_{S,V}^\lam\,\,\,\, (\rm mod\,\,\, A(<\lam)),$$
where $\Phi(T,U)\in R$ depends only on $T$ and $U$. Thus for a cell
module $W_C(\lam)$, we can define a symmetric bilinear form $\Phi
_{\lam}:\,\,W_C(\lam)\times W_C(\lam)\longrightarrow R$ by
$$\Phi_{\lam}(C_{S},C_{T})=\Phi(S,T).$$
The radical of the bilinear form is defined to be
$$\rad_{\Phi}\lam:= \{x\in W(\lam)\mid \Phi_{\lam}(x,y)=0\,\,\,\text{for all} \,\,\,y\in W_C(\lam)\}.$$
By the general theory of cellular algebras, $\rad_{\Phi}\lam$ is an
$A$-submodule of $W_C(\lambda)$, motivating the definition
$L(\lambda)=W_C(\lam)/\rad_{\Phi}\lam$. In principle, the next
result of Graham and Lehrer classifies the simple $A$-modules.

\begin{lem} {\rm(\cite{GL})}\label{2.6}
Let $K$ be a field and $A$ a finite dimensional cellular algebra over $K$. Let
$\Lambda_{0}=\{\lam\in\Lambda\mid \Phi_{\lam}\neq 0\}$. Then
$\{L(\lam)\mid \lam\in\Lambda_{0}\}$ is a complete set of pairwise non-isomorphic absolutely irreducible $A$-modules.
\end{lem}

For any $\lam\in\Lambda$, fix an order on $M(\lam)$ and let
$M(\lam)=\{S_{1},S_{2},\cdots,S_{n_{\lam}}\}$, where $n_{\lam}$ is
the number of elements in $M(\lam)$, the matrix
$$G(\lam)=(\Phi(S_{i},S_{j}))_{1\leq i,j\leq n_{\lam}}$$ is called
{\em Gram matrix}. Note that all the determinants of $G(\lam)$
defined with different order on $M(\lam)$ are the same. By the
definition of $G(\lam)$ and $\rad_{\Phi}\lam$, for a finite
dimensional cellular algebra $A$, if $\Phi_{\lam}\neq 0$, then
$\dim_{K}L(\lam)=\rank G(\lam)$.

\section{Projective cell modules of Frobenius cellular algebras}
\label{xxsec3}

Let $K$ be a field. Let $A$ be a finite dimensional $K$-algebra and
$M$ an $A$-module. The algebra homomorphism
$$\rho_{M}:A\rightarrow
 \End_{K}(M),\,\,\, \rho_{M}(a)m=am,\,\,\,\, \,\,\,\forall\,\, m\in M,\,\,\,\, a\in A,$$ is
called the representation afforded by $M$.

\smallskip

Let $A$ be a finite dimensional Frobenius cellular $K$-algebra with
a cell datum $(\Lambda, M, C, i)$. Given a non-degenerate bilinear
form $f$,  denote the left dual basis by $D=\{D_{S,T}^\lam \mid
S,T\in M(\lam),\lam\in\Lambda\}$, which satisfies
$$
\tau(D_{U,V}^{\mu}C_{S,T}^{\lam})=\delta_{\lam\mu}\delta_{SV}\delta_{TU}.
$$
Denote the right dual basis by $d=\{d_{S,T}^\lam \mid S,T\in
M(\lam),\lam\in\Lambda\}$, which satisfies
$$
\tau(C_{S,T}^{\lam}d_{U,V}^{\mu})=\delta_{\lam\mu}\delta_{S,V}\delta_{T,U}.
$$
For $\mu\in\Lambda$, let $A_{d}(>\mu)$ be the $K$-subspace of $A$
generated by
$$\{d_{X,Y}^\epsilon \mid X,Y\in M(\epsilon),\mu<\epsilon\}$$ and let $A_{D}(>\mu)$ be
the $K$-subspace of $A$ generated by
$$\{D_{X,Y}^\epsilon \mid X,Y\in M(\epsilon),\mu<\epsilon\}.$$
Note that the $K$-linear map $\alpha$ which sends $d_{X,Y}^\epsilon$
to $D_{X,Y}^\epsilon$ is a Nakayama automorphism of the algebtra
$A$.

For arbitrary $\lam, \,\mu\in \Lambda$, $S,\,T\in M(\lam)$,
$U,\,V\in M(\mu)$, write
$$C_{S,T}^{\lam}C_{U,V}^{\mu}=\sum\limits_{\epsilon\in\Lambda,X,Y\in M(\epsilon)}
r_{(S,T,\lam),(U,V,\mu),(X,Y,\epsilon)}C_{X,Y}^{\epsilon},$$

$$D_{S,T}^{\lam}D_{U,V}^{\mu}=\sum\limits_{\epsilon\in\Lambda,X,Y\in M(\epsilon)}
R_{(S,T,\lam),(U,V,\mu),(X,Y,\epsilon)}D_{X,Y}^{\epsilon}.$$
Applying $\alpha^{-1}$ on both sides of the above equation, we
obtain
$$d_{S,T}^{\lam}d_{U,V}^{\mu}=\sum\limits_{\epsilon\in\Lambda,X,Y\in M(\epsilon)}
R_{(S,T,\lam),(U,V,\mu),(X,Y,\epsilon)}d_{X,Y}^{\epsilon}.$$

The following lemma about structure constants will play an important
role in determining the projective cell modules of Frobenius
cellular algebras.

\begin{lem}\label{3.1}
For arbitrary $\lam,\mu\in\Lambda$ and $S,T,P,Q\in M(\lam)$, $U,V\in
M(\mu)$ and $a\in A$, the following hold:
\begin{enumerate}
\item[(1)] $D_{U,V}^{\mu}C_{S,T}^{\lam}=\sum\limits_{\epsilon\in
\Lambda, X,Y\in
M(\epsilon)}r_{(S,T,\lam),(Y,X,\epsilon),(V,U,\mu)}D_{X,Y}^{\epsilon}.$
\item[(2)] $D_{U,V}^{\mu}C_{S,T}^{\lam}=\sum\limits_{\epsilon\in
\Lambda, X,Y\in
M(\epsilon)}R_{(Y,X,\epsilon),(U,V,\mu),(T,S,\lam)}C_{X,Y}^{\epsilon}.$
\item[(3)] $aD_{U,V}^{\mu}\equiv \sum\limits_{U'\in
M(\mu)}r_{i(\alpha^{-1}(a))}(U,U')D_{U',V}^{\mu}\,\,\,(\mod
A_{D}(>\mu))$.
\smallskip
\item[(4)] $D_{P,Q}^{\lam}C_{S,T}^{\lam}=0\,\,\,\, if \,\,\,Q\neq S.$
\smallskip
\item[(5)] $D_{U,V}^{\mu}C_{S,T}^{\lam}=0 \,\,\,\,if
\,\,\,\mu\nleq \lam.$
\smallskip
\item[(6)] $D_{T,S}^{\lam}C_{S,Q}^{\lam}=D_{T,P}^{\lam}C_{P,Q}^{\lam}.$
\smallskip
\item[(7)] $C_{S,T}^{\lam}d_{U,V}^{\mu}=\sum\limits_{\epsilon\in
\Lambda, X,Y\in
M(\epsilon)}r_{(Y,X,\epsilon),(S,T,\lam),(V,U,\mu)}d_{X,Y}^{\epsilon}.$
\item[(8)] $C_{S,T}^{\lam}d_{U,V}^{\mu}=\sum\limits_{\epsilon\in
\Lambda, X,Y\in
M(\epsilon)}R_{(U,V,\mu),(Y,X,\epsilon),(T,S,\lam)}C_{X,Y}^{\epsilon}.$
\item[(9)] $d_{U,V}^{\mu}a\equiv \sum\limits_{V'\in
M(\mu)}r_{\alpha(a)}(V,V')d_{U,V'}^{\mu}\,\,\,\,\,\,\,\,\,\,\,(\mod
A_{d}(>\mu))$.
\smallskip
\item[(10)] $C_{S,T}^{\lam}d_{P,Q}^{\lam}=0\,\, if \,\,T\neq P.$
\smallskip
\item[(11)] $C_{S,T}^{\lam}d_{U,V}^{\mu}=0 \,\,\,\,if\,\,\, \mu\nleq \lam.$
\smallskip
\item[(12)] $C_{S,T}^{\lam}d_{T,P}^{\lam}=C_{S,Q}^{\lam}d_{Q,P}^{\lam}.$
\end{enumerate}
\end{lem}

\begin{proof}
(1), (3), (4), (5), (7), (8), (9), (11) have been obtained in
\cite{L3}. (2), (8) are proved similarly as (1), (7), respectively.
(6) is a direct corollary of (1) and Definition \ref{2.4}. Finally,
(12) is obtained similarly as (6).
\end{proof}

It follows from Definition \ref{2.4} and Lemma \ref{3.1} (1), (3)
that for arbitrary elements $S,T,P,Q\in M(\lam)$,
$$D_{S,T}^\lam D_{P,Q}^\lam \equiv \Psi(T,P)D_{S,Q}^\lam\,\,\,\,
(\rm mod\,\,\, A_D(>\lam)),$$ where $\Psi(T,P)\in K$ depends only on
$T$ and $P$. Applying $\alpha^{-1}$ on both sides of the above
equation, we obtain
$$d_{S,T}^\lam d_{P,Q}^\lam \equiv \Psi(T,P)d_{S,Q}^\lam\,\,\,\,
(\rm mod\,\,\, A_d(>\lam)).$$ Take an order on $M(\lam)$ the same as
in the definition of $G(\lam)$. Then the matrix $(\Psi(S, T))$ will
be denoted by $G'(\lam)$.

Let $W_C(\lam)$ be a cell module. Define $\varphi_{ST}\in {\rm
End}_K(W_C(\lam))$ by $\varphi_{ST}(C_X)=\delta_{SX}C_T$ for
arbitrary $S, T\in M(\lam)$. The following lemma reveals a relation
among $I(\varphi_{ST})$, $\Phi_{\lam}$ and $\Psi_{\lam}$.

\begin{lem}\label{3.2}
If $W_C(\lam)$ is simple, then $I(\varphi_{ST})=c_{ST}{\rm
id}_{W_C(\lam)},$ where $$c_{ST}=\sum\limits_{Q\in
M(\lam)}\Phi(T,Q)\Psi(S,Q).$$
\end{lem}

\begin{proof}
It follows from $W_C(\lam)$ being simple and Schur's Lemma that
$$I(\varphi_{ST})=c_{ST}{\rm id}_{W_C(\lam)},\eqno(3.1)$$ where $c_{ST}\in K$.
Let $\rho_{\lam}$ be the representation afforded by
$W_C(\lam)$. Then for $a\in A$, we have $$aC_X=\sum\limits_{Y\in
M(\lam)}\rho_{\lam}(a)_{YX}C_Y.$$ A direct computation gives
that
$$I(\varphi_{ST})(C_X)=\sum_{\eta\in \Lambda, \,P, \,Q\in M(\eta),
\,\,Y\in M(\lambda)}\rho_{\lam}(D_{Q,
P}^{\eta})_{SX}\rho_{\lam}(C_{P,Q}^{\eta})_{YT}C_Y.\eqno(3.2)$$
Combining (3.1) and (3.2) yields that $$\sum_{\eta\in \Lambda, P,
Q\in M(\eta), Y\in M(\lambda)}\rho_{\lam}(D_{Q,
P}^{\eta})_{SX}\rho_{\lam}(C_{P,Q}^{\eta})_{XT}=c_{ST}.\eqno(3.3)$$
By the definition of cell modules, $$C_{P,
Q}^{\eta}C_{T}=r_{C_{P,Q}^{\eta}}(X,T)C_X+\sum\limits_{X'\neq
X}r_{X'}C_{X'},$$ where $r_{C_{P,Q}^{\eta}}(X,T)\in K$ is defined in
Definition \ref{2.1}(C3) and $r_{X'}\in K$. This implies that
$$\rho_{\lam}(C_{P,Q}^{\eta})_{XT}=r_{(P, Q, \eta),(T,T,\lam),(X,T,\lam)}.\eqno(3.4)$$
On the other hand, it follows from Lemma \ref{3.1}(2) that
$$\rho_{\lam}(D_{Q, P}^{\eta})_{SX}=R_{(S,S,\lam),(Q, P, \eta),(S,X,\lam)}.\eqno(3.5)$$
We have from (3.3), (3.4) and (3.5) that
$$c_{ST}=\sum\limits_{\eta\in\Lambda, P,Q\in M(\eta)}r_{(P, Q, \eta),(T,T,\lam),(X,T,\lam)}R_{(S,S,\lam),(Q, P, \eta),(S,X,\lam)}.\eqno(3.6)$$
Again by Lemma \ref{3.1}, for any $\epsilon\in\Lambda$, if
$\epsilon\nleq\lam$, then
$R_{(S,S,\lam),(Y,X,\epsilon),(S,S,\lam)}=0$, if
$\lam\nleq\epsilon$, then
$r_{(X,Y,\epsilon),(S,S,\lam),(S,S,\lam)}=0$. Thus Definition
\ref{2.4} and Lemma \ref{3.1} force (3.6) being
\begin{eqnarray*}
c_{ST}&=&
\sum\limits_{P, Q\in M(\lam)}r_{(P, Q, \lambda),(T,T,\lam),(X,T,\lam)}R_{(S,S,\lam),(Q, P, \lambda),(S,X,\lam)}\\
&=&\sum\limits_{Q\in M(\lam)}r_{(X,Q,\lam),(T,T,\lam),(X,T,\lam)}R_{(S,S,\lam),(Q,X,\lam),(S,S,\lam)}\\
&=&\sum\limits_{Q\in M(\lam)}\Phi(T,Q)\Psi(S,Q).
\end{eqnarray*}
We complete the proof.
\end{proof}

Now we are able to describe the main result of this section.

\begin{thm}\label{3.3}
Let $A$ be a finite dimensional Frobenius cellular $K$-algebra and
$W_C(\lam)$ a simple cell module. Then $W_C(\lam)$ is projective if
and only if there exist $S, \,T\in M(\lam)$ such that $\Psi(S,T)\neq
0$.
\end{thm}

\begin{proof}
It follows from Corollary \ref{2.3} that $W_C(\lam)$ is projective
if and only if there exist $X, Y\in M(\lam)$ such that
$I(\varphi_{XY})\neq 0$. By Lemma \ref{3.2}, this is equivalent to
$c_{XY}\neq 0$. Fix an order on $M(\lam)$ and denote the matrix
$(c_{XY})$ by $I_{\lam}$. Then $W_C(\lam)$ is projective if and only
if $I_{\lam}\neq 0$. On the other hand, we have from Lemma \ref{3.2}
that $G'(\lam)G(\lam)=I_{\lam}$. The simplicity of $W_C(\lam)$
implies that $G(\lam)$ is invertible. Thus $I_{\lam}\neq 0$ if and
only if $G'(\lam)\neq 0$, that is, there exist $S, \,T\in M(\lam)$
such that $\Psi(S,T)\neq 0$.
\end{proof}

We can obtain from this theorem a necessary condition for a cell
module being projective.

\begin{cor}\label{3.4}
If $G'(\lam)=0$, then $W_C(\lam)$ is not a projective module.
\end{cor}

\begin{proof}
The corollary follows from Theorem \ref{3.3} and \cite[Corollary
1.2]{KX8} clearly.
\end{proof}

\begin{remark}\label{3.5}
Using the right dual basis $\{d_{X,Y}^{\epsilon}\mid
\epsilon\in\Lambda, X, Y\in M(\epsilon)\}$, for each $\lam\in
\Lambda$ we can define an $A$-module $W_d(\lam)$ as follows. As a
$K$-basis, $W_d(\lam)$ has a basis $\{d_S\mid S\in M(\lam)\}$. The
$A$-action is defined by $$ad_S=\sum_{S'\in
M(\lam)}r_{i(a)}(S,S')d_{s'}.$$ Then by a similar way, we can prove
that if $W_d(\lam)$ is simple, then it is projective if and only if
$\lam\in\Lambda_0$.
\end{remark}

Now let us apply Theorem \ref{3.3} to an example which was
constructed by K\"{o}nig and Xi in \cite{KX2}.

\begin{example}
Let $K$ be a field. Let us take $\lam\in K$ with $\lam\neq 0$ and
$\lam\neq 1$. Let $$A=K\langle a, b, c, d\rangle/I,$$ where $I$ is generated by
$$a^2, b^2, c^2, d^2, ab, ac, ba, bd, ca, cd, db, dc, cb-\lam bc,
ad-bc, da-bc.$$ Let $\Lambda=\{1, 2, 3\}$. If we define $\tau$ by
$\tau(1)=\tau(a)=\tau(b)=\tau(c)=\tau(d)=0$ and $\tau(bc)=1$ and
define an involution $i$ on $A$ to be fixing $a$ and $d$, but
interchanging $b$ and $c$,  then $A$ is a Frobenius cellular algebra
with a cellular basis
\[ \begin{matrix}
\begin{matrix} bc \end{matrix} ;&
\begin{matrix} a & b\\ c &
d\end{matrix} \,\,; &
\begin{matrix} 1 \end{matrix}.
\end{matrix} \]
The left dual basis is
\[ \begin{matrix}
\begin{matrix} 1 \end{matrix} ;&
\begin{matrix} d & c/\lam\\ b &
a\end{matrix}\,\, ; &
\begin{matrix} bc \end{matrix}.
\end{matrix} \]
It is easy to check that $W_C(1)=0$, $W_C(2)=0$ and $W_C(3)$ is
simple. However, $G'(3)=0$. This implies that $W_C(3)$ is not
projective. Thus none of cell modules is projective.
\end{example}

\bigskip

\section{A special case}
\label{xxsec4}

Throughout this section, we assume that $A$ is a finite dimensional
Frobenius cellular algebra with both the dual bases being cellular
with respect to the opposite order on $\Lambda$.

Under this assumption, we obtained the following result about the
Nakayama automorphism $\alpha$ in \cite{L3}.

\begin{lem}\cite[Theorem 3.1]{L3}\label{4.1}
For arbitrary $\lam\in \Lambda$ and $S, T\in M(\lam)$, we have
$$\alpha(C_{S,T}^{\lam})\equiv C_{S,T}^{\lam}\quad\quad (\mod
A(<\lam)).$$
\end{lem}

In order to generalize the so-called Schur elements to Frobenius
case, we prove a lemma first.

\begin{lem}\label{4.2}
Let $A$ be a Frobenius cellular algebra with cell datum $(\Lambda,
M, C, i)$. For every $\lam\in\Lambda$ and $S,T\in M(\lam)$, we have
$$D_{S,T}^{\lam}C_{T,S}^{\lam}D_{S,T}^{\lam}C_{T,S}^{\lam}=\sum_{S'\in
M(\lam)}\Phi(S',T)\Psi(S',T)D_{S,T}^{\lam}C_{T,S}^{\lam}.$$
\end{lem}
\begin{proof}
By Lemma \ref{3.1} (3), (5), we have
\begin{eqnarray*}
D_{S,T}^{\lam}C_{T,S}^{\lam}D_{S,T}^{\lam}C_{T,S}^{\lam}
&=&D_{S,T}^{\lam}(C_{T,S}^{\lam}D_{S,T}^{\lam})C_{T,S}^{\lam}\\&=&\sum\limits_{S'\in
M(\lam)}r_{i(\alpha^{-1}(C_{T, S}^{\lam}))}(S, S')D_{S,
T}^{\lam}D_{S', T}^{\lam}C_{T, S}^{\lam}.
\end{eqnarray*}
It follows from Definition~\ref{2.4} and Lemma~\ref{4.1} that
$$r_{i(\alpha^{-1}(C_{T, S}^{\lam}))}(S, S')=\Phi(S', T).$$
Again by Lemma \ref{3.1}(5), the desired equation follows.
\end{proof}

We have from Lemma \ref{3.1} that $D_{S,T}^{\lam}C_{T,S}^{\lam}$ is
independent of $T$. Then for any $\lam\in\Lambda$, we can define a
constant $k_{\lam}$ as follows.
\begin{dfn}\label{4.3}
For $\lam\in\Lambda$, take an arbitrary $T\in M(\lam)$. Define
$$k_{\lam}=\sum\limits_{X\in M(\lam)}\Phi(X,T)\Psi(X,T).$$
\end{dfn}

The following lemma reveals the relation among $G(\lam)$, $G'(\lam)$
and $k_{\lam}$. This result is a bridge which connects $k_{\lam}$
with $c_{ST}$.

\begin{lem}\label{4.4}
Let $\lam\in\Lambda$. Fix an order on the set $M(\lam)$. Then
$G(\lam)G'(\lam)=k_{\lam}E$, where $E$ is the identity matrix.
\end{lem}

\begin{proof}
Clearly, we only need to show that $\sum\limits_{X\in
M(\lam)}\Phi(X,S)\Psi(X,T)=0$ for arbitrary $S, T\in M(\lam)$ with
$S\neq T$.

Let us consider
$D_{S,S}^{\lam}C_{S,S}^{\lam}D_{S,S}^{\lam}C_{T,S}^{\lam}$. On one
hand, it is zero by Lemma \ref{3.1}. On the other hand, computing it
as that in Lemma \ref{4.2} yields that
$$D_{S,S}^{\lam}C_{S,S}^{\lam}D_{S,S}^{\lam}C_{T,S}^{\lam}=\sum\limits_{X\in
M(\lam)}\Phi(X,S)\Psi(X,T)D_{S,S}^{\lam}C_{S,S}^{\lam}.$$ It follows
from $D_{S,S}^{\lam}C_{S,S}^{\lam}\neq 0$ that $\sum\limits_{X\in
M(\lam)}\Phi(X,S)\Psi(X,T)=0.$
\end{proof}

\begin{cor}\label{4.5}
Keep notations as above, then $c_{ST}={\rm
Tr}(\varphi_{ST})k_{\lam}$, where ${\rm Tr}$ denotes the usual
matrix trace.
\end{cor}

\begin{proof}
Note that ${\rm Tr}(\varphi_{ST})=\delta_{ST}$. Then by Lemma
\ref{3.2}, Definition \ref{4.3} and Lemma \ref{4.4}, the corollary
follows.
\end{proof}

The constants $k_{\lam}$ could be viewed as generalizations of Schur
elements when $W_C(\lam)$ is simple. We are in a position to give
the main result of this section. It is similar to that in symmetric
case.
\begin{thm}\label{4.6}
Let $A$ be a Frobenius cellular algebra with
$i(D_{S,T}^{\lam})=D_{T,S}^{\lam}$ and
$i(d_{S,T}^{\lam})=d_{T,S}^{\lam}$ for arbitrary $\lam\in \Lambda$
and $S,T\in M(\lam)$. Then the cell module $W_C(\lam)$ is projective
if and only if $k_\lam\neq 0$.
\end{thm}
\begin{proof}
By \cite[Corollary 1.2]{KX2}, if $W_C(\lam)$ is projective, then
$W_C(\lam)$ is irreducible. It follows from Corollary \ref{2.3} that
there exist $S, T\in M(\lam)$ such that $I(\varphi_{ST})\neq 0,$ or
$c_{ST}\neq 0$. This implies that $k_{\lam}\neq 0$ by Corollary
\ref{4.5}.

Conversely, assume that $k_{\lam}\neq 0$. Then $W_C(\lam)$ is
irreducible by Lemma \ref{4.4}. Again by Corollary \ref{2.3},
$W_C(\lam)$ is projective.
\end{proof}

\begin{cor}
$W_C(\lam)$ is projective if and only if so is $W_d(\lam)$.
\end{cor}

\begin{proof}
By Theorem \ref{4.6}, we only need to prove that $W_d(\lam)$ is
projective if and only if $k_{\lam}\neq 0$. In fact, if
$k_{\lam}\neq 0$, then Lemma \ref{4.4} forces $W_d(\lam)$ to be
simple. Thus $W_d(\lam)$ being projective follows. Conversely, if
$W_d(\lam)$ is projective, then \cite[Corollary 1.2]{KX2} implies
that $W_d(\lam)$ is simple, that is, $G'(\lam)$ is invertible.
Moreover, we have from Remark \ref{3.5} that $G(\lam)\neq 0$. Hence
we conclude by Lemma \ref{4.4} that $k_{\lam}\neq 0$.
\end{proof}

\begin{remark}
Using the left dual basis $D_{X,Y}^{\epsilon}$, we can also define
modules $W_D(\lam)$. It could be proved that $W_D(\lam)$ is
projective if and only if so is $W_d(\lam)$. We omit the details
here.
\end{remark}

\end{document}